\documentclass[11pt]{article}

\usepackage{amsmath,amssymb,amsthm,mathtools}
\usepackage{booktabs,longtable,array}
\usepackage[margin=1in]{geometry}
\usepackage{xcolor}
\usepackage{listings}
\usepackage[hidelinks]{hyperref}
\usepackage{hyperref}

\newtheorem{proposition}{Proposition}[section]
\newtheorem{lemma}[proposition]{Lemma}
\theoremstyle{remark}
\newtheorem{remark}[proposition]{Remark}

\newcommand{\conv}{*}
\newcommand{\dd}{\,d}
\newcommand{\abs}[1]{\lvert#1\rvert}

\newcommand{\code}[1]{\texttt{#1}}
\newcommand{\Tmark}{\textnormal{\textsc{T}}}
\newcommand{\Cmark}{\textnormal{\textsc{C}}}
\newcommand{\Cpmark}{\textnormal{\textsc{Cp}}}
\newcommand{\Csmark}{\textnormal{\textsc{Cs}}}
\newcommand{\hash}[2]{\texttt{#1}\newline\texttt{#2}}
\newcolumntype{L}[1]{>{\raggedright\arraybackslash}p{#1}}
\definecolor{codegray}{rgb}{0.96,0.96,0.96}
\title{Computer-assisted Proof Under Audit:\\
Typos, Certificate Errors, and Reproducible Exact Checks\\
for a Symbolic Invertibility Proof}
\author{Fan Zheng}

\begin{document}
\maketitle

\begin{abstract}
We audit the exact-arithmetic certificate in arXiv:2310.19781v2, its 2025
\emph{Communications in Mathematical Physics} version, and their shared
MATLAB archive.  CMP corrects none of the audited items; we therefore treat
both versions together and use the openly accessible arXiv source for exact
provenance.  We found 11 proof-affecting defects (alongside typographical slips)
and give reproducible exact counterexamples to several implemented bounds and
reconstruction steps.  To the best of our knowledge, this is the first
independent, version-pinned, source-level audit of a published computer-assisted
proof (CAP) in analysis to identify multiple proof-affecting defects in its
released computational certificate.  The findings do not refute the intended theorem or its analytic
reduction, but they show that the published certificate does not prove the
claimed conclusion.  Future work will give a corrected, structurally simpler
proof with substantially less machine assistance.
\end{abstract}

\begin{center}
\begin{minipage}{0.91\textwidth}
\small
\textbf{Declaration of AI use.}
The preparation of this note used OpenAI GPT-5.6 Sol through Codex, accessed in
August 2026, for literature survey, source
navigation, comparison of formulas and code, exact-arithmetic scripts,
drafting, and \LaTeX{} editing.  Model output was not treated as mathematical
evidence: each reported item is verified against a pinned source location and, where
numerical, against a reproducible exact check.  The human author takes
responsibility for the final text and its mathematical claims.

\textbf{Acknowledgements.} The author is supported by Ram\'on y Cajal grant
RYC2022-035363-I. He is also grateful to the Feishu Institute for hosting
his visit there.
\end{minipage}
\end{center}

\tableofcontents
\clearpage

\section{Introduction}
\label{sec:introduction}

\paragraph{What is a computer-assisted proof?}
A computer-assisted proof (CAP) is a proof whose conclusion depends on at
least one claim established by a machine-executed computation: for example,
an exhaustive search, an exact symbolic calculation, a rigorous interval
enclosure, or the checking of a proof object.  The output is not
self-authenticating.  The proof includes the
whole chain from mathematical statement to computational specification, from
specification to implementation, and from output back to conclusion.  Because
the computation is indispensable, checking only the prose does not check the
whole proof.

Questions of surveyability, trust, rigorous numerics, and reproducibility are
developed in a substantial literature
\cite{Tymoczko1979,Lanford1987,KochSchenkelWittwer1996,Tucker2005,
Neumaier2002,Rump2010,vanDenBergLessard2015,GomezSerrano2019,
RevolTheveny2014}.

\paragraph{What is an audit?}
An audit independently examines that chain for a pinned paper version and set
of artifacts.  It is more than a rerun, which can faithfully reproduce the
same defect.  A source-level audit asks whether the code computes the quantity
defined in the paper, whether its predicate implies the stated inequality,
whether the required domain is covered, and whether errors and dependencies
are propagated correctly.  It also separates a printed slip already treated
correctly by the executed argument from a defect on which the certificate
depends.  Formal proof offers a complementary model in which an explicit proof
object can reduce the implementation-level trust boundary
\cite{AvigadHarrison2014,HalesEtAl2017}.

\paragraph{Why audit?}
An audit presupposes neither success nor failure; it tests whether the released
evidence establishes the claim.  If it confirms the proof, it replaces
untested confidence by independently checked support within its scope.  If it
reveals errors, it locates where the certificate fails and what must be
repaired.  Both outcomes are valuable in different ways: confirmation
strengthens warranted confidence, while error detection enables correction.
Here the audit finds multiple proof-relevant defects despite exact arithmetic
used in the certificate.

\paragraph{Earlier audits and the novelty here.}
The contribution here is a version-pinned, source-level audit that aligns a
published argument in analysis with its released symbolic certificate and
later refereed version, separates print-only slips from proof-affecting
defects, and supplies independently checkable witnesses for the latter.

Independent scrutiny has important precedents.  Schmidt's 1982 partial audit
of the Appel--Haken proof found mishandled and possibly omitted cases, some
requiring changes to the discharging procedure \cite{Schmidt1982}.  Later work
provided an author-led repair of the Kepler proof, partially audited Flyspeck,
formally recertified Tucker's Lorenz computation, and revisited Lam's problem
with checkable SAT certificates
\cite{HalesRevision2010,Adams2016,Immler2018,BrightEtAl2021}.  The last study
found inconsistent enumeration counts while preserving the nonexistence
conclusion.  Most adverse precedents located in our search concern
finite-enumeration or hybrid certificates; Immler's analytic precedent
recertified the Lorenz computation without reporting multiple defects in its
released certificate.  Among
the independent audits located in our search, this appears to be the first of
this kind to identify multiple proof-affecting defects in a published
computational certificate supporting a theorem in analysis, thus extending
the error-detection capability of the existing CAP auditing methodology
to a new field where it has great potential applications.

\paragraph{Combinatorial and analytic/numerical audits.}
They share the same standard of rigor, with emphasis on different aspects.
A combinatorial audit emphasizes faithful encoding, exhaustive case
or symmetry coverage, and a sound discrete checker.  In an
analytic/numerical CAP, finite data often certify a claim about a continuous or
infinite-dimensional object.  Its audit may additionally have to control
domain coverage, discretization and truncation, tails, directed enclosure or
order propagation, and the translation of formulas into code.  Merely
confirming that the program ran or that its integer arithmetic was internally
consistent does not check those interfaces
\cite{BrightEtAl2021,vanDenBergLessard2015,GomezSerrano2019,Immler2018}.

\paragraph{Why CAP audits require a stringent standard.}
The validity standard is the same for every proof, but CAPs warrant especially
explicit standards of traceability, reproducibility, and independent checking.
A local gap in a traditional proof may still leave its conceptual argument viable
and admit a transparent repair.  One wrong computed endpoint can instead
invalidate every dependent enclosure.  Numerical tolerance is legitimate only
when enclosed and propagated; exact fractions remove roundoff, interval
arithmetic absorbs only declared uncertainty, and hashes identify an artifact.
None repairs a wrong constant, an omitted region, a reversed inequality, or a
broken implication.

This enlarged interface also broadens the referee's burden.  Computer output
is often expected to be reliable, yet CAP defects can be harder to detect than
a line-level gap.  Review may require simultaneous command of the mathematics,
numerical analysis, interval semantics, programming language, and toolchain,
placing one referee near the practical limit of what a person can survey.
Reproducible artifacts, independent implementations, and formal proof objects
can reduce or distribute that burden but cannot erase the trust boundary
\cite{Lanford1987,KochSchenkelWittwer1996,RevolTheveny2014,
AvigadHarrison2014,HalesEtAl2017}.

\paragraph{Why audit this certificate.}
The audited certificate fits in the broader and fashionable program on singularity
formation in incompressible fluids.  It is intended to prove invertibility,
modulo a one-dimensional kernel, for a linearized operator around an
approximate self-similar Boussinesq profile, an input to the companion analysis
of the blow-up of two-dimensional Boussinesq and three-dimensional Euler equations
\cite{EP-main,EP-symbolic}.  Thus the certificate plays a consequential analytic
role in the whole proof, yet Its computational scope remains confined and focused.

\paragraph{The present audit.}
The audit is forensic but constructive: it separates evident printed slips
from defects that invalidate an executed check or required implication, gives
short exact counterchecks, and traces downstream dependencies.  Publication in
CMP changed none of the audited items and supplied no replacement companion
file.  Our conclusions concern this specified symbolic certificate, not the
analytic reduction or the truth of the intended theorem; the aim is to state
what must change for itt to become a rigorous proof.

\section{Scope, version, and terminology}
\label{sec:scope}

The audited article exists in two forms: \cite{EP-symbolic,EP-CMP}:
arXiv:2310.19781v2 [math.AP], submitted on 8 April 2025, and the subsequent
2025 \emph{Communications in Mathematical Physics} version of record.  We
checked the complete PDFs of both versions.  The journal article corrects none
of the mathematical items audited in v2: it retains the relevant printed
formulas and identifiers and continues to rely on the same arXiv companion
files without providing a replacement code archive.  The mathematical
findings are therefore stated for v2 and CMP together.

Detailed page, equation, source-line, and code-line locations remain pinned to
v2 because the executable source provenance is the unmodified archive
distributed with that version.  This choice also keeps the audit readily
accessible: the arXiv PDF and source archive are openly available, whereas
access to the complete publisher version may depend on a reader's subscription
or institution.  Unless explicitly marked ``CMP,'' page and equation numbers
refer to the v2 PDF; the corresponding CMP locations are collected in
Section~\ref{sec:cmp-check}.  Source-line numbers are one-based physical lines
in the pinned arXiv archive.  The archive and the two PDFs used for this audit
have the following SHA--256 digests:
\begin{description}
\small
\item[Source archive.]
\nolinkurl{748370E0C93484190D777EDC35DFCE0CE4747787B0EBFB425C72F0C8F83851FE}
\item[v2 PDF.]
\nolinkurl{126D4D004F13C4107B1D69E800CDD3DD98282CB6F6042B207FFF0FB9FF43AC9C}
\item[CMP version-of-record PDF.]
\hash{DD7EB8017F4064987F5300FC405657677}{3501647B8F8C16C4B484AB5367391A0}
\end{description}

We use two classifications.
\begin{description}
\item[Printed typo (\Tmark).]
The printed formula is wrong or inconsistent, but the intended correction is
already used by the surrounding proof or by the released code, and no
independent numerical certificate relies on the misprint.
\item[Certificate error (\Cmark).]
The released proof or code fails to establish an implication needed by the
certificate.  A one-character programming mistake is placed in this class if
the released computation actually relies on it. Within this category we further
distinguish \textbf{errors already present in the paper (Cp)} from \textbf{errors only appearing in the source code (Cs)}.
\end{description}
This convention separates the seriousness of the consequence from the likely
cause of the slip.  In particular, calling an item a certificate error does
not suggest that the intended theorem is false; it means that the certificate
released in v2 and retained in CMP needs a corrected argument or recomputation.

For clarity, the phrase \emph{v2 prints} refers to the PDF, \emph{the released
code} refers to the source archive just identified, and \emph{a corrected
calculation} describes what is mathematically required rather than what v2
already does.

\subsection{Check against the published CMP version}
\label{sec:cmp-check}

The paper subsequently appeared in \emph{Communications in Mathematical
Physics}, volume 406, article 261, and was published online on 3 October 2025
\cite{EP-CMP}.  We checked the complete 36-page version-of-record PDF whose
digest is recorded above, rather than relying on the publisher's HTML preview.
We compared every entry in Table~\ref{tab:classification} with the journal
text.  No audited printed item is corrected.  The identifiers (16), (27),
(30), (32), (33), and~\((\star)\) are retained; only their page locations
change.  The journal also continues to direct the reader to the companion
files at the arXiv submission and provides no replacement code archive.
Since CMP is the refereed version, this is also the observable outcome of the
refereeing and publication process: none of the audited items was corrected.
The referee reports are not public, so this statement concerns the published
outcome and does not speculate about what any individual referee may have
noticed.
Neither the publisher page nor Crossref's relation metadata linked a
correction or erratum as of 13 August 2026.\footnote{The relation metadata can
be queried at
\url{https://api.crossref.org/works/10.1007/s00220-025-05367-6}.}

For the main-body items, the complete PDF gives the following direct
cross-check:
\begin{itemize}
\item Lemma~2.2 still starts the displayed truncations at \(j=1\) despite its
  nonzero zero mode (CMP p.~12), while the later working truncation still
  starts at \(j=0\) (CMP p.~26).  Lemma~2.4 still states the mode-25 residual
  claim (CMP p.~13).
\item Equation~(27), the overlapping six-region decomposition~\((\star)\),
  and the missing rings in the modified-weight series remain on CMP
  pp.~20--21.  Equation~(30) retains both disputed signs on CMP p.~21.
\item The local reversed inequality in the proof of Lemma~5.1 remains on CMP
  p.~22.  The transformed-profile and beta/eta formula slips remain in
  Lemmas~5.4--5.8 and their proofs on CMP pp.~23--27.
\end{itemize}
The Appendices and companion-code interface give the same conclusion:
\begin{itemize}
\item CMP Appendix A.3.2 (p.~29) retains the slot convention
  \(\code{cring(j)}=\widehat{\mathring c}_{j-1}\) and still says that its
  Section B3 computes \(P_{N_4},Q_{N_4}\), while referring to the same
  companion code.  Thus publication does not repair the mode-25 selection or
  the coefficient-sign loop.
\item CMP Appendix B.1 (pp.~31--32) retains both the \(M=N=5000\) tail claim
  and the line \(110+3=113\).  Appendix B.2 (p.~33) retains the extra
  \(2/\pi\), the factor
  \(c_N^{(j-1)}\), and the domain \(j\ge0\); hence the zeroth-moment problem is
  unchanged.
\item CMP Appendix A.2.3 (p.~28) still describes equation~(30) using
  \(-45.6\) as a lower bound for the Laplace datum.  Appendices A.3.6--A.3.7
  (p.~30) still delegate the directed Picard iterates and derivative stepping
  to the same \path{shorttime_sym.m}.  The code-only endpoint, coefficient-loop,
  Picard-order, omitted-block, point-value, and discounted-tail findings are
  therefore unchanged in status.
\item Print-only items also persist: Appendix A.3.7 still names \(P_1,P_2\),
  Appendix A.3.8 (p.~31) still prints an eta integrand without \(\gamma^2\),
  and Lemma~D.1 (p.~35) still uses the wrong dummy index in its proof.
\end{itemize}
Although journal production changes the pagination and adds publication front and back
matter, it does not change any mathematical item audited here.  Detailed
locations below remain keyed to v2 for source-code reproducibility; the CMP
locations in this subsection record the independent version-of-record check.

\section{Minimal setup}
\label{sec:setup}

This section isolates only the definitions needed to understand the audit.
The full analytic reduction belongs to \cite{EP-main,EP-symbolic}.  To make
the provenance of the finite formulas explicit, each group below identifies
whether it is copied from the paper, algebraically rewritten from it, or
introduced as local audit notation.  Unless ``CMP'' is written, the equation
and page references in this section are to v2.  Numbers attached to displays
in this note are local; an original equation number is always identified
explicitly as belonging to v2 or CMP.

\subsection{Modified dual weights and the terminal forcing}

Set
\[
 x=x(\gamma):=\frac1{1+\gamma^2}.
\]
For a scalar angular function, use the normalized average
\[
 \langle h\rangle:=\frac2\pi\int_0^\infty
 \frac{h(\gamma)}{1+\gamma^2}\dd\gamma.
\]
Here \(x\) merely abbreviates the powers written explicitly in v2
Lemma~2.2 (p.~11; CMP p.~12).  The average is the
\(\gamma=\tan\theta\) form of the paper's normalized angular average
(v2 pp.~4, 7; CMP pp.~4, 7) \cite{EP-symbolic,EP-CMP}.

For bookkeeping, encode the time-dependent coefficient data in the formal
generating series below.  This is a compact synthesis of the modified-weight
definition in v2 equations~(13)--(15) (pp.~10--11), the coefficient
construction in Lemma~2.2 (p.~11), and the working truncations in
Section~5.3 (p.~23); CMP retains these ingredients on pp.~11--12 and~26.
The v2 paper does not display these infinite generating series, and no
convergence assertion for them is used here \cite{EP-symbolic,EP-CMP}:
\begin{equation}
 \mathcal W_1(t,x):=\mathring c_0(t)+\sum_{j\ge1}\mathring c_j(t)x^j,
 \qquad
 \mathcal W_2(t,x):=\sum_{j\ge1}\mathring d_j(t)x^j.
 \label{eq:weight-series}
\end{equation}
Inverse Laplace transformation of the displayed recurrence in v2 Lemma~2.2
gives a triangular ODE recurrence for the positive-index coefficients.  The
only part needed below is that the coefficient of the first omitted power
after truncation through mode \(M\) is, verbatim from v2 Lemma~2.2
(p.~11; \path{invertibility.tex}, lines 445--449; CMP p.~12)
\cite{EP-symbolic,EP-CMP},
\begin{equation}
 P_M=8M\mathring c_M+10\mathring d_M,
 \qquad
 Q_M=8M\mathring d_M.
 \label{eq:terminal-forcing}
\end{equation}
Thus the correctly indexed finite sums are, in notation introduced here,
\begin{equation}
 C_M=\sum_{j=0}^M\mathring c_jx^j,\qquad
 D_M=\sum_{j=0}^M\mathring d_jx^j,
 \label{eq:weight-truncation}
\end{equation}
where \(\mathring d_0=0\).  The inclusion of \(j=0\) is not cosmetic: the
modified first weight has the nonzero zero mode computed in
\eqref{eq:zero-mode}.  This lower
limit is the one used by the paper's later working truncations
(v2 p.~23; \path{invertibility.tex}, lines 1053--1060; CMP p.~26), rather
than the inconsistent lower limit in the statement of Lemma~2.2
\cite{EP-symbolic,EP-CMP}.

The released \path{shorttime_sym.m}, lines 18--37, first generates modes
\(1,\ldots,M\) as arrays \code{c(1:M)}, \code{d(1:M)}, and then prepends the
zero mode.  Consequently
the packed arrays obey
\begin{equation}
 \code{cring(r)}=\widehat{\mathring c}_{r-1},\qquad
 \code{dring(r)}=\widehat{\mathring d}_{r-1}.
 \label{eq:slot-convention}
\end{equation}
This one-based convention is stated in v2 Appendix~A.3.2 (p.~26;
\path{invertibility.tex}, lines 1185--1192) and retained in CMP
Appendix~A.3.2 (p.~29) \cite{EP-symbolic,EP-CMP}.

\subsection{Volterra iteration and directed envelopes}

With \(f=\Upsilon\) and \(K=K_2\), the short-time equation below is the
algebraically rearranged form of equation~(18) in v2 Lemma~2.5 (p.~13),
repeated as equation~(31) on
p.~20; the corresponding CMP locations are pp.~14 and~22
\cite{EP-symbolic,EP-CMP}.  The convolution symbol is local shorthand:
\begin{equation}
 f=g-K\conv f,\qquad
 (h\conv k)(t):=\int_0^t h(t-s)k(s)\dd s.
 \label{eq:volterra}
\end{equation}
Its Picard iterates are those of v2 Lemma~5.2 and Remark~5.3 (p.~19;
CMP p.~22) \cite{EP-symbolic,EP-CMP}.  We rename the paper's \(P_m\) as
\(\Pi_m\) to avoid collision
with the terminal forcing \(P_M\) in \eqref{eq:terminal-forcing}:
\begin{equation}
 \Pi_0=g,\qquad \Pi_{m+1}=g-K\conv\Pi_m.
 \label{eq:picard}
\end{equation}
The next display is audit notation packaging the kernel and source bounds
described in v2 Appendices~A.3.3 and A.3.5 (p.~26; CMP pp.~29--30)
\cite{EP-symbolic,EP-CMP}; in particular, \(E\) is introduced here.  Suppose
the available directed bounds are
\begin{equation}
 0\le K_\ell\le K\le K_u,\qquad
 g_\ell\le g\le g_u,\qquad E:=K_u-K_\ell\ge0.
 \label{eq:envelopes}
\end{equation}
Because convolution by a nonnegative kernel preserves order, choosing the
correct endpoint in a product depends on the sign of the function being
convolved.  This elementary point is the source of the Picard issue in
Section~\ref{sec:picard-error}.

\subsection{Profile coefficients and weighted moments}

Put \(\beta:=\gamma^2/(1+\gamma^2)=1-x\); this is local shorthand for the
factor written explicitly in the paper, not the radial variable denoted there
by \(z\).  In the paper's 13-rescaled convention, let
\(\bar\Gamma_i:=\Gamma_i^*-\Gamma_i^*(0)\).  V2 Lemma~3.2 (p.~14) gives
the profile series, initially from \(k=0\), and proves \(a_0=b_0=0\);
CMP reproduces it as Lemma~3.2 on p.~15 \cite{EP-symbolic,EP-CMP}.  Hence
\begin{equation}
 \bar\Gamma_1=\sum_{k\ge1}a_k\beta^k,\qquad
 \bar\Gamma_2=\sum_{k\ge1}b_k\beta^k.
 \label{eq:profile-series}
\end{equation}
The rational recurrence is v2 equation~(24) (p.~15;
\path{invertibility.tex}, lines 611--638), retained as CMP equation~(24)
on p.~16 \cite{EP-symbolic,EP-CMP}:
\begin{equation}
 \binom{a_k}{b_k}=W_k\binom{a_{k-1}}{b_{k-1}}\quad(k\ge2),\qquad
 W_k=
 \begin{pmatrix}
 \dfrac{32k^2+4k-101}{2(16k^2+30k-19)}&
 \dfrac{26(k-1)}{16k^2+30k-19}\\[5pt]
 -\dfrac{35}{16k^2+30k-19}&
 \dfrac{16k^2-4k-12}{16k^2+30k-19}
 \end{pmatrix}.
 \label{eq:profile-recurrence}
\end{equation}
Remark~3.1 of v2 gives \(\mathsf A=-351/19\), while the proof of
Lemma~3.2 gives \(b_1=(35/27)\mathsf A\) and
\(a_1=(13/14)b_1\) (p.~14; CMP pp.~15--16).  Simplifying those printed
identities yields the normalized seed \cite{EP-symbolic,EP-CMP}:
\begin{equation}
 a_1=-\frac{845}{38},\qquad b_1=-\frac{455}{19}.
 \label{eq:profile-seed}
\end{equation}
The moment integral occurs in the proof of v2 Lemma~4.4 (p.~17) and again
in Appendix~B.2 (p.~29); CMP retains it on pp.~19 and~32--33.  The
factorial expression below is the exact half-integer-gamma simplification of
the paper's displayed beta-integral value \cite{EP-symbolic,EP-CMP}:
\begin{equation}
 c_k^{(j)}=\frac2\pi\int_0^\infty
 \left(\frac{\gamma^2}{1+\gamma^2}\right)^k
 \frac1{(1+\gamma^2)^j}\frac{\dd\gamma}{\gamma^2}
 =\frac{(2k-2)!(2j)!}
 {4^{k+j-1}(k-1)!j!(k+j-1)!}
 \qquad(k\ge1,\ j\ge0),
 \label{eq:moment-coefficient}
\end{equation}
Substitution of the absolutely convergent profile series into the integral
definitions of Appendix~B.2 gives
\begin{equation}
 I_j=\sum_{k\ge1}a_kc_k^{(j)},\qquad
 J_j=\sum_{k\ge1}b_kc_k^{(j)}.
 \label{eq:profile-moments}
\end{equation}
These definitions already include the normalization \(2/\pi\).  Thus
\eqref{eq:profile-moments} is an exact series rewrite of the paper's
integrals, not a separately printed definition (v2 Appendix~B.2, p.~29;
CMP pp.~32--33) \cite{EP-symbolic,EP-CMP}.

The coefficient checks below also use the two constants
\[
 \mathsf A:=-\frac{351}{19},\qquad c_*:=\frac{237}{46},
\]
The v2 paper calls the second constant \(c\).  The value \(\mathsf A\) is
printed in Remark~3.1 (p.~14; CMP p.~15), and \(c=237/46\) is used in the modified
decomposition on v2 p.~10 and Section~4.1 p.~16 (CMP pp.~10, 17)
\cite{EP-symbolic,EP-CMP}.
The transformed and modified profile coefficients are
\begin{align}
 \mathrm{af}_k&=-\frac12a_k+\frac{13}{2}b_k,\nonumber\\
 \mathrm{bf}_1&=5a_1+5\mathsf A-\frac12b_1,&
 \mathrm{bf}_k&=5(a_k-a_{k-1})-\frac12b_k\quad(k\ge2),\nonumber\\
 \mathrm{afm}_k&=-\left(c_*+\frac12\right)a_k+\frac{13}{2}b_k,\nonumber\\
 \mathrm{bfm}_1&=5a_1+5\mathsf A-\left(c_*+\frac12\right)b_1,&
 \mathrm{bfm}_k&=5(a_k-a_{k-1})-\left(c_*+\frac12\right)b_k
 \quad(k\ge2).
 \label{eq:transformed-coefficients}
\end{align}
These are the quantities denoted \code{af}, \code{bf}, \code{afm}, and
\code{bfm} in the released scripts.  V2 equation~(32) names the unmodified
\(\mathrm{af},\mathrm{bf}\) series (p.~22; CMP p.~25).  Their formulas follow
by inserting \eqref{eq:profile-series} into the initial datum \(F^*\) following
v2 equations~(5)--(6) (p.~6; \path{invertibility.tex}, lines 252--261;
CMP p.~7) and using the profile equations~(20)--(21) (v2 p.~14; CMP p.~15).
The proof
of Lemma~4.4 prints the modified \(\mathrm{afm},\mathrm{bfm}\) formulas
(v2 pp.~17--18; \path{invertibility.tex}, lines 782--793; CMP pp.~19--20)
\cite{EP-symbolic,EP-CMP}.

\subsection{The datum at \texorpdfstring{\(\xi=-1\)}{xi=-1} and the filtered average}

For one counterexample it is convenient to denote the unmodified positive
weight modes at \(\xi=-1\) by \(p_j,q_j\).  This notation is introduced
here: the display is the scalar \(\xi=-1\) specialization
\(p_j=c_j(-1)\), \(q_j=d_j(-1)\) of v2 Lemma~2.2
(pp.~11--12; \path{invertibility.tex}, lines 412--435 and 453--472),
retained in CMP Lemma~2.2 on pp.~11--12 \cite{EP-symbolic,EP-CMP}:
\begin{align}
 p_0&=\frac19,&q_0&=\frac{11}{54},\nonumber\\
 p_1&=\frac{10q_0}{17},&q_1&=\frac{13p_1}{20},\nonumber\\
 p_j&=\frac{8(j-1)p_{j-1}+10q_{j-1}}{8j+9},&
 q_j&=\frac{8(j-1)q_{j-1}+13p_j}{8j+12}\quad(j\ge2).
 \label{eq:minus-one-recurrence}
\end{align}
All entries in this recurrence are positive.

At large time, v2 reconstructs a filtered angular average.  It writes
\(\Theta^{\rm mod}:=\Theta-c\Gamma^*\), with \(c=237/46\), on p.~10
(\path{invertibility.tex}, lines 361--365), and normalizes the 13-rescaled
profile by \(\langle\Gamma_1^*\rangle=26\) on p.~6, footnote~10; CMP
retains these formulas on pp.~10 and~7.  Writing \(c_*=c\), the first
component therefore becomes \cite{EP-symbolic,EP-CMP}
\[
 \Theta_1=c_*\Gamma_1^*+\Theta_1^{\rm mod}.
\]
For brevity, introduce the local shorthand
\[
 T(t):=\langle\Theta_1^{\rm mod}(t)\rangle.
\]
Its Laplace value is the quantity written out in the v2 proof overview
(p.~7) and estimated in equation~(29)/Lemma~4.5 (pp.~18--19), retained on
CMP pp.~8 and~21 \cite{EP-symbolic,EP-CMP}.  With the Laplace convention
used after v2 equation~(15) (p.~11; CMP p.~11),
\[
 \widehat T(-1)=\int_0^\infty e^sT(s)\dd s.
\]
The finite-time filter is the 13-rescaled form of the footnote to v2
equation~(1) (p.~6; \path{invertibility.tex}, lines 244--248), after
inserting the decomposition above.  The same rescaled convention is used in
the overview (p.~7) and in equation~(30) (p.~19); CMP retains the
corresponding displays on pp.~6, 8, and~21 \cite{EP-symbolic,EP-CMP}.  Thus
the following is an exact algebraic rewrite of the defining filter, not a
quotation of equation~(30):
\begin{equation}
 \Upsilon(t)=26c_*e^{-t}+T(t)-e^{-t}\int_0^te^sT(s)\dd s.
 \label{eq:filter-start}
\end{equation}
Here \(\Upsilon\) denotes the later, \(13\)-rescaled convention adopted in v2
after its initial definition.  Multiplication by this positive constant does
not affect any sign conclusion.  Splitting the finite integral at infinity
gives the sign-correct terminal form in \eqref{eq:correct-duhamel}; v2/CMP
equation~(30) instead prints the two reversed signs audited there
\cite{EP-symbolic,EP-CMP}.

\section{Printed typos}
\label{sec:typos}

The items in this section are local misprints under the classification in
Section~\ref{sec:scope}.  Some printed statements are literally false, but
the surrounding argument or code already uses the evident intended version.

\subsection{Lower limit in the finite weight}
\label{sec:lower-limit}

Lemma~2.2 as printed in v2 (PDF p.~11; \path{invertibility.tex}, lines
436--449) defines its finite first-weight sum from \(j=1\).  Yet the displayed
Laplace coefficients in the same lemma give
\[
 c_0(\xi)=\frac1{\xi+10},\qquad
 d_0(\xi)=\frac{\xi+23}{(\xi+13)(\xi+10)},
\]
and therefore
\begin{equation}
 \widehat{\mathring c}_0=\frac{c_0}{d_0}-1=-\frac{10}{\xi+23},\qquad
 \mathring c_0(t)=-10e^{-23t}\ne0.
 \label{eq:zero-mode}
\end{equation}
If that mode is omitted, the purported tail has initial value \(-10\), not
the zero initial datum used later.

This is best classified as a statement typo rather than a separate
certificate error.  Equation~(16) in the proof of Lemma~2.2
(\path{invertibility.tex}, lines 455--456) is not itself the definition of
the modified truncation: it is the underlying unmodified weight expansion,
and that expansion begins at \(j=0\).  More decisively, the later working
modified-weight truncation on v2 p.~23 (source line 1059) \emph{currently}
begins at \(j=0\).  Neither of those two formulas needs this correction; they
identify \(j=1\) in the statement of Lemma~2.2 as the inconsistent lower
limit.  The corrected statement is \eqref{eq:weight-truncation}.

\subsection{Coefficient names in the transformed-profile tail}
\label{sec:coef-name}

Equation~(32) on v2 p.~22 expands two transformed quantities \(H_1,H_2\)
using coefficients denoted \(\mathrm{af}_k,\mathrm{bf}_k\).  Lemma~5.8
immediately afterward defines its tail size using the original profile
coefficients \(a_N,b_N\).  As printed, the bound is false: at \(N=4\), direct
use of the displayed recurrence gives
\[
 (a_4^2+b_4^2)^{1/2}<\frac{31}{10},\qquad
 \abs{\mathrm{af}_5}>\frac{24}{5},\qquad
 \abs{\mathrm{bf}_5}>5.
\]
Thus the first omitted transformed coefficient already exceeds the printed
constant.  The released code, however, uses
\code{abs(af(h))} and \code{abs(bf(h))} componentwise in
\path{shorttime_sym.m}, lines 105--109.  We therefore classify the names
\(a_N,b_N\) in the lemma statement as a printed typo.  This classification
does not assert that the separate componentwise tail estimate is proved; it
only says that this particular mismatch is not what the code implements.

The same distinction resolves two nearby slips on v2 p.~23.  Source lines
1061--1066 assert \(\mathrm{bf}_j\le0\) for \(j=0,\ldots,25\), although mode
zero is not defined and \eqref{eq:profile-recurrence}--
\eqref{eq:transformed-coefficients} give
\[
 \mathrm{bf}_2=\frac{1222}{19}>0.
\]
The released implementation does not use the false uniform sign claim:
\path{shorttime_sym.m}, lines 98--104, branches separately on the sign of each
\code{bf(h)}.  We therefore classify this as a printed derivation error, not
an independent implementation error.  Also, source lines 1053--1056 write
\(\ell_N\) after choosing the two cutoffs \(N_3,N_4\); the transformed-profile
tail cutoff there is \(N_3\), so the symbol should be \(\ell_{N_3}\).

\subsection{Normalization and endpoint slips}
\label{normalization-error}

The unnumbered moment-tail display before Lemma~B.2 on v2 p.~29 places an
extra \(2/\pi\) in front of \(I_j^{>N},J_j^{>N}\), although the definitions
\eqref{eq:moment-coefficient}--\eqref{eq:profile-moments} already contain
that normalization.  The released code adds its tolerance directly to
\(I_j,J_j\), so the extra factor is a printed normalization typo.  A separate,
proof-relevant failure of that tolerance is treated in
Section~\ref{sec:moment-error}.

The same Appendix gives the recurrence for \(c_k^{(j)}\) on source lines
1349--1351 with the range \(j>0,\,k\ge0\).  But \(c_0^{(j)}\) lies outside
\eqref{eq:moment-coefficient} (and its defining integral diverges at the
origin); the recurrence range is \(k\ge1\).  The released arrays start at
\(k=1\), so this is another print-only index typo.

The rectangular decomposition \((\star)\) on v2 p.~18 lets its finite block
(1) reach \(j=35\) while blocks (2) and (4) also begin at \(j=35\), so the
printed regions overlap.  At least one endpoint must change.  The code range
\code{1:35} represents \(j=0,\ldots,34\), suggesting one possible convention;
a complete disjoint partition must also account for \(k>35,j=35\).  This
endpoint inconsistency is a printed typo.  The omission of whole finite blocks
from the code is a separate certificate error in
Section~\ref{sec:rectangles}.

The same page has a second notation-only mismatch
(\path{invertibility.tex}, lines 795--799).  The left side pairs \(H_i\) with
the modified weights \(\widehat{\mathring w}_i\), while the displayed double
series uses the unmodified symbols \(c_j,d_j\) instead of
\(\mathring c_j,\mathring d_j\).  This is literally wrong, especially at the
zero mode, but \path{profile.m}, lines 147--148 and 183--186, uses
\code{cringp}, \code{dringp}.  We therefore classify the missing rings as a
printed notation typo.

\subsection{Beta-tail notation and derivation typos}
\label{sec:beta-tail}

The differential inequality on v2 p.~21 (source line 929) drops the exponent
\(M+1\) from its factor
\((1+e^{-8t}\gamma_0^2)^{-(M+1)}\).  The tail forcing in Lemma~2.2, the
integrated formula at source line 933, and \path{shorttime_sym.m} all retain
that exponent.  This is a printed exponent typo.

In the proof of the beta-tail bound on v2 p.~22
(\path{invertibility.tex}, line 1001), the first integral after splitting at
\(\gamma=1\) prints \((1+e^{8r})^{-M}\), omitting the factor
\(\gamma^2\) from that denominator; it also omits the terminal \(d\gamma\)
from that first integral.  The next line restores
\((1+e^{8r}\gamma^2)^{-M}\), but source line 1004 then displays the weaker and
incompatible second term \(e^{-4r}/(2M-1)\).  The stated bound and
\path{shorttime_sym.m}, lines 305 and 308, instead use
\(e^{-8Mr}/(2M-1)\).  That stated estimate is obtained directly, for
\(\gamma\ge1\), from
\[
 \frac{1}{(1+e^{8r}\gamma^2)^M(1+\gamma^2)}
 \le e^{-8Mr}\gamma^{-2M}.
\]
Integration over \([1,\infty)\) gives
\(e^{-8Mr}/(2M-1)\).  Thus the missing \(\gamma^2\) at line 1001 and the
\(e^{-4r}\) term at line 1004 are defects in the printed derivation, while the
statement and code use an independently valid bound.  We therefore classify
this as a derivation typo rather than a certificate error.

Appendix~A.3.8 on v2 p.~27 repeats the same missing factor in its prose
description of \code{eta}: source lines 1243--1246 print the denominator
\((1+e^{8t})^M\), with no \(\gamma^2\).  The actual definition on v2 p.~21,
the statement of the beta-tail lemma, and \path{shorttime_sym.m}, lines
295--310, all concern \((1+e^{8t}\gamma^2)^M\).  We classify the Appendix
description as the same printed notation typo.

Two nearby index statements are also print-only.  First, Lemma~5.6 on v2
p.~22 (source lines 970--983) permits \(N=0\).  But the definition
\eqref{eq:moment-coefficient}'s analogue for \(\eta(r,M,N)\) has integrand
asymptotic to \(\gamma^{2N-2}\) at the origin, so \(\eta(r,M,0)\) diverges;
the displayed \(c_0\) is not defined either.  The range must be \(N\ge1\).
The released calls use \(n=1\) or \(n=h\in\{1,2,3,4\}\), so no computed
enclosure relies on \(N=0\).

Second, Appendix~A.3.4 on v2 p.~26 (source lines 1211--1216) swaps the roles
of the arguments of \code{alphabeta}.  The function definition at
\path{shorttime_sym.m}, lines 316--326, calls \code{eta(m+1,n,C)}; the calls
at lines 65 and 90 use \code{m=N4} for the weight cutoff and \(n=1\) or
\(n=h\) for the power of \(\gamma^2/(1+\gamma^2)\).  Thus equation~(33)
should display the power \(n\), while its footnote should say \(m=N_4=25\).
The code uses those intended roles, so this is an Appendix notation typo, not
an additional implementation error.

Finally, v2 p.~23 (source line 1080) writes
\(\eta_u(8(t-w),N_4+1,k)\).  Since the definition of \(\eta(r,M,N)\) already
contains \(e^{8r}\), while the denominator immediately above is
\(1+e^{8(t-w)}\gamma^2\), the first argument there must be \(t-w\), not
\(8(t-w)\).  The substitution in \path{shorttime_sym.m}, lines 318--320, is
indeed \code{tt-ss}.  This is another printed argument typo not executed by
the code.

\subsection{Further local prose and notation slips}
\label{sec:local-prose}

In the proof of Lemma~5.1 on v2 p.~20, the displayed relation
\[
 f=g-K_2\conv f\ge g
\]
must read \(f\le g=P_0\) by \eqref{eq:volterra} and nonnegativity, as the
surrounding alternating inequalities already
require.  In Appendix~D.1 on v2 p.~31 (source line 1432), the summation dummy
variable should be \(j\):
\[
 \sum_{j=M_2}^{k-1}\log(1-\alpha/j)
 \le-\alpha\sum_{j=M_2}^{k-1}\frac1j.
\]
Finally, Appendix~A.3.7 on v2 p.~27 (source line 1241) says that the polynomialized
iterates are \(P_1,P_2\); \path{shorttime_sym.m}, lines 245--250, constructs
\(P_1,P_3\), consistently with the preceding paragraph.  These three slips
are typographical and are not independent numerical tests.

There are also three harmless index/norm slips.  In Lemma~2.2 on v2 p.~11
(source lines 436--440), the family \((\mathring c_j,\mathring d_j)\) is first
subscripted by \(k\), and the following condition says \(k\ge1\); both indices
are \(j\).  The phrase ``whenever \(M\in\mathbb N,\,N\ge2\)'' introduces an
otherwise unused \(N\); the intended condition is on \(M\).  Finally, the
display on v2 p.~28 (source line 1265) omits the subscript \(\infty\) from its
first remainder norm.  The surrounding formulas fix all three readings, so
they are print-only notation typos.

\section{Short-time coefficient and Picard certificate}
\label{sec:short-errors}

\subsection{The upper endpoint is a hybrid mode}
\label{sec:endpoint-shift}

\begin{proposition}[Mode-25 endpoint shift]\label{prop:endpoint-shift}
The coefficient recurrence correctly generates modes \(1,\ldots,25\).
The subsequent array selection uses mode-24 values with the mode-25
multiplier, and the finite sums omit mode 25.
\end{proposition}

\begin{proof}
Lines 23--30 of \path{shorttime_sym.m} correctly generate
\code{c(k)}, \code{d(k)} for \(1\le k\le25\).  Lines 33--37 then prepend the
zero mode, so \eqref{eq:slot-convention} holds.  In particular, slot 25 is
mode 24 and slot 26 is mode 25.

Nevertheless, with \code{N4=25}, lines 57--58 execute
\begin{lstlisting}
P = 8*N4*ilaplace(cring(N4)) + 10*ilaplace(dring(N4));
Q = 8*N4*ilaplace(dring(N4));
\end{lstlisting}
and therefore form
\begin{equation}
 P_{\rm code}=200\mathring c_{24}+10\mathring d_{24},\qquad
 Q_{\rm code}=200\mathring d_{24}.
 \label{eq:hybrid-terminal}
\end{equation}
By \eqref{eq:terminal-forcing}, the intended pair is
\begin{equation}
 P_{25}=200\mathring c_{25}+10\mathring d_{25},\qquad
 Q_{25}=200\mathring d_{25},
 \label{eq:true-terminal}
\end{equation}
stored in slot 26.  The executed pair is not \(P_{24},Q_{24}\) either,
because those would carry the multiplier \(8\cdot24=192\).

Lines 62--63 and 79--80 use ranges ending at \code{N4}.  Under
\eqref{eq:slot-convention}, every included entry is the correct coefficient
of its own mode, but the range is only \(0,\ldots,24\); mode 25 alone is
omitted.  Meanwhile \code{alphabeta} at line 318 uses
\code{eta(m+1,...)} and hence the exponent \(26\), which belongs to a
truncation through mode 25.  Lines 34--35 of \path{PQ_sym.m} repeat the same
off-by-one terminal selection.  The corrected endpoint and ranges should be
\code{N4+1} (or \code{M+1}).
\end{proof}

\begin{remark}[Scope of the error]
This is not a claim that every modal coefficient in the released array is
wrong.  The recurrence has correctly computed the modes; the endpoint
selection is wrong.  However, the hybrid \(P_{\rm code},Q_{\rm code}\) is then
used as the forcing in the tail construction.  Consequently the resulting
tail, kernel, source, and Picard enclosures must be recomputed.  Correcting
only a displayed endpoint after those downstream objects have been generated
would not repair the certificate.  In particular, the released calculation
does not establish Lemma~2.4 for the actual \(P_{25}-Q_{25}\).
\end{remark}

\subsection{The residual coefficient loop is not componentwise}
\label{sec:sign-loop}

After constructing a polynomial coefficient vector
\code{bigcoeffsimple}, \path{PQ_sym.m}, lines 52--65, loops over an index
\code{j}.  The intended test is a sign assertion for each coefficient.
The two predicates actually executed on lines 57 and 61 are
\begin{lstlisting}
bigcoeffsimple <= 0
bigcoeffsimple >= 0
\end{lstlisting}
with no \code{(j)}.  Thus the loop variable is absent from the condition, and
the loop does not test the claimed scalar statement at index \code{j}.
In fact, according to MATLAB documentation
(\url{https://www.mathworks.com/help/matlab/ref/if.html}),
the statements enclosed in an \code{if} block will execute if the condition contains
\textit{only} nonzero values. Thus a block of positive coefficients followed by
another block of negative coefficients will \textit{never} pass the if test, so
the flag indicating a potential wrong sign of the coefficients will never be raised.
Lines 73--77 nevertheless split the vector at the asserted single sign change to
form \code{pospoly} and \code{negpoly}, pretending that the sign check has passed.

The local code correction is to test \code{bigcoeffsimple(j)}.  We do not
need to rely on any particular MATLAB rule for interpreting a vector-valued
condition: the decisive fact is that the executed predicate is not the
componentwise predicate claimed in the comment.  This appears to be a
one-token oversight, but it is a certificate error because no independent
componentwise audit is supplied and the subsequent decomposition uses the
unverified sign pattern.  This observation does not assert that the intended
sign pattern is false, however.

\subsection{Sign-sensitive Picard multiplication}
\label{sec:picard-error}

\begin{lemma}[Safe second-iterate upper bound]\label{lem:safe-picard}
Assume \eqref{eq:envelopes}, let \(L_1\le\Pi_1\), and suppose
\(L_1\ge-\Lambda\).  Then
\begin{equation}
 \Pi_2\le
 g_u-K_\ell\conv L_1+\Lambda E\conv1.
 \label{eq:safe-picard}
\end{equation}
Here \(1\) denotes the constant-one function on \([0,\infty)\).
\end{lemma}

\begin{proof}
By the recurrence \eqref{eq:picard}, and since \(K\ge0\) and
\(\Pi_1\ge L_1\),
\[
 \Pi_2=g-K\conv\Pi_1\le g_u-K\conv L_1.
\]
Writing \(K=K_\ell+(K-K_\ell)\) gives
\[
 -K\conv L_1
 =-K_\ell\conv L_1-(K-K_\ell)\conv L_1.
\]
The bounds \(L_1\ge-\Lambda\) and
\(0\le K-K_\ell\le E\) imply
\[
 -(K-K_\ell)\conv L_1\le\Lambda E\conv1,
\]
which proves \eqref{eq:safe-picard}.
\end{proof}

Lines 235--242 of \path{shorttime_sym.m} construct directed Picard iterates.
At the second upper step, line 241 uses only
\begin{equation}
 U_2^{\rm old}=g_u-K_\ell\conv L_1.
 \label{eq:old-picard}
\end{equation}
The width term \(\Lambda E * 1\) in \eqref{eq:safe-picard} is absent.  Formula
\eqref{eq:old-picard} is valid if \(L_1\ge0\), but the released script checks
the first lower iterate only through \(\log3\) (lines 256--274) while using
the third iterate through \(\log4\) (lines 276--285).  It provides no
\(L_1\ge0\) hypothesis on the additional interval.

This is an order error, not merely a missing safety margin.  If \(L_1<0\),
then multiplication by a larger kernel makes \(K\conv L_1\) more negative,
so \(-K\conv L_1\) is larger.  The lower kernel endpoint therefore gives a
lower, not an upper, contribution from that part.  The failure of the claimed
order implication already appears in the scalar model
\[
 K=1,\qquad K_\ell=0,\qquad L_1=-1,\qquad g_u=0.
\]
Then the true contribution \(g_u-KL_1\) equals \(1\), whereas
\eqref{eq:old-picard} gives \(0\).  Thus nonnegativity of \(L_1\), or an
explicit replacement for it, is essential.  The logically sufficient repair
is \eqref{eq:safe-picard}; it requires both a lower bound \(-\Lambda\) for
\(L_1\) and the kernel width \(E\).

\section{Profile bounds and moments}
\label{sec:profile-errors}

\subsection{The recorded \texorpdfstring{\(L^\infty\)}{L-infinity} test is five times too weak}
\label{sec:Loo-weak}

Appendix~B.1 of v2 (PDF pp.~27--28;
\path{invertibility.tex}, lines 1259--1266) derives
\begin{equation}
 \|E_N^i\|_\infty
 \le10M^{1.1}N^{-0.1}\sqrt{a_M^2+b_M^2}
 \label{eq:v2-linfty-bound}
\end{equation}
and then chooses \(M=N=5000\).  Substitution into
\eqref{eq:v2-linfty-bound} and squaring its nonnegative sides show that the
claimed 1/2 estimate would require
\begin{equation}
 100N^2(a_N^2+b_N^2)<\frac14.
 \label{eq:needed-linfty-test}
\end{equation}
Line 38 of \path{profile.m} instead forms
\[
 \rho_N:=20N^2(a_N^2+b_N^2),
\]
and line 39 tests only \(\rho_N<1/4\).  That condition is a factor of five
weaker than \eqref{eq:needed-linfty-test}.

This is an insufficient verification, not merely a presentational mismatch.
Exact application of \eqref{eq:profile-recurrence}--\eqref{eq:profile-seed}
to the released \(N=5000\) row gives the small outward bracket
\begin{equation}
 \frac3{20}<\rho_{5000}<\frac4{25}.
 \label{eq:rho-bracket}
\end{equation}
Thus the implemented test \(\rho_{5000}<1/4\) succeeds, while the square of
the bound printed in both v2 and CMP is \(5\rho_{5000}>3/4\), not \(<1/4\).
The calculation in the script therefore does not prove
\(\|E_N^i\|_\infty<1/2\) from \eqref{eq:v2-linfty-bound}.  We make no claim
here that a sharper, different profile argument cannot prove such a bound.

\subsection{The subsequent 113 budget does not follow}
\label{sec:budget-error}

There is an independent arithmetic error in the next \(L^\infty\) estimate
on v2 p.~28 (\path{invertibility.tex}, lines 1288--1291).  Even if one grants
the two asserted bounds
\(\|E_N^1\|_\infty,\|E_N^2\|_\infty\le1/2\), the displayed contribution to
the second transformed component is
\[
 5\|E_N^1\|_\infty+
 \left(\frac{237}{46}+\frac12\right)\|E_N^2\|_\infty
 \le\frac52+\frac{65}{23}
 =\frac{245}{46}>5.
\]
V2 replaces this contribution by \(3\) in the calculation
\(110+3=113\).  Since \(113\) is then used as an input to the comparison
budget, and is hardcoded in \path{longtime_sym.m}, lines 23 and 25, this is a
certificate error rather than a harmless arithmetic typo in an unused
display.

\subsection{The zeroth moment has no valid tail allowance}
\label{sec:moment-error}

For \(j\ge1\), positivity and the beta integral give the useful identity
\[
 \sum_{k>N}c_k^{(j)}=c_{N+1}^{(j-1)}.
\]
At \(j=0\), the analogous positive sum diverges; in particular,
\(c_N^{(-1)}\) is not defined by \eqref{eq:moment-coefficient}.  Nevertheless,
the unnumbered estimate immediately before Lemma~B.2 on v2 p.~29 claims a
bound using \(c_N^{(j-1)}\) while explicitly including \(j=0\).  After the
normalization typo identified in Section~\ref{sec:typos} is removed, the
printed formula still supplies no zeroth-row estimate.

The released code uses a different quantity.  After resetting \(N=500\) in
\path{profile.m}, lines 77--79, lines 96--109 test
\begin{equation}
 (a_N^2+b_N^2)c_N^{(0)}<(2\cdot10^{-4})^2
 \label{eq:endpoint-proxy}
\end{equation}
and then add \(2\cdot10^{-4}\) as an error allowance to every row.
\path{build_shorttime_data.m}, lines 55--67, repeats the same step.
Expression \eqref{eq:endpoint-proxy} is an endpoint proxy, not either tail
sum in \eqref{eq:profile-moments}; neither v2 nor CMP gives an implication from
the former to the latter.

\begin{proposition}[Counterexample with \(j=0\)]\label{prop:moment-counterexample}
For the rational profile generated by
\eqref{eq:profile-recurrence}--\eqref{eq:profile-seed},
\[
 I_0^{>500}:=\sum_{k>500}a_kc_k^{(0)}>2\cdot10^{-4}.
\]
Thus the uniform tolerance inserted by the released code is false for the
zeroth moment.
\end{proposition}

\begin{proof}
The recurrence for the moment coefficient is especially short:
\[
 c_1^{(0)}=1,\qquad
 c_{k+1}^{(0)}=\frac{2k-1}{2k}c_k^{(0)}.
\]
Exact rational propagation of this recurrence and
\eqref{eq:profile-recurrence} gives the following integer lower witnesses:
\[
\begin{array}{c|rrrrrrrrrrr}
k&501&502&503&504&505&506&507&508&509&510&511\\ \hline
10^6a_k>&786&783&781&778&776&774&771&769&766&764&762\\
10^6c_k^{(0)}>&25225&25199&25174&25149&25124&
25099&25075&25050&25025&25001&24976
\end{array}
\]
Each entry is certified by substituting the preceding exact rational pair
into the displayed recurrences and clearing a positive denominator.  In
particular, throughout this eleven-row block,
\[
 a_k>\frac{750}{10^6}=\frac3{4000},\qquad c_k^{(0)}>\frac4{165}.
\]
The eleven displayed positive terms alone therefore give
\[
 I_0^{>500}>
 \sum_{k=501}^{511}a_kc_k^{(0)}
 >11\frac3{4000}\frac4{165}
 =\frac1{5000}=2\cdot10^{-4}.
\]
\end{proof}

\begin{remark}
The proposition is a direct counterexample to the tolerance used by the code,
not only a complaint that the printed \(j=0\) formula is undefined.  The
remaining terms of the tail are unnecessary for the contradiction.
\end{remark}

\section{The \texorpdfstring{\(\xi=-1\)}{xi=-1} pairing and rectangular bookkeeping}
\label{sec:minus-one-errors}

\subsection{A direct counterexample to equation (27) in v2 and CMP}
\label{sec:equation-27-error}

Equation~(27) on v2 p.~18 (\path{invertibility.tex}, lines 804--815) claims
\begin{equation}
 \sum_{j\ge N}c_k^{(j)}\abs{d_j(-1)}
 \le\frac32\frac{N+k-1}{k+3/2}\,c_k^{(N)}m_N,\qquad
 m_N=\max\{\abs{c_N(-1)},\abs{d_N(-1)}\}.
 \label{eq:v2-horizontal}
\end{equation}
There is already a warning in the preceding asymptotics.  The two bounds
printed immediately before that claim combine as
\[
 j^{-0.85}j^{-k+1/2}=j^{-k-0.35}.
\]
An integral comparison at that decay scale involves \(k-0.65\), not the
printed \(k+1.5\).

The inequality itself, not only its derivation, has a short exact
counterexample.  At \(\xi=-1\), the positive-index modified coefficients used
in \path{profile.m} differ from the unmodified coefficients
\eqref{eq:minus-one-recurrence} by the same positive factor \(54/11\).
That factor cancels from both sides of \eqref{eq:v2-horizontal}.  Put
\[
 h_j:=c_1^{(j)}=\frac{\binom{2j}{j}}{4^j}.
\]
For \(N=10,k=1\), the coefficient in front of \(h_{10}m_{10}\) in
\eqref{eq:v2-horizontal} is \(6\).  Exact evaluation of
\eqref{eq:minus-one-recurrence} gives
\[
 m_{10}=p_{10}
 =\frac{295375930834093}{15979148319214080}
 <\frac{37}{2000},\qquad
 h_{10}=\frac{46189}{262144}.
\]
Consequently the asserted right side is
\begin{equation}
 6h_{10}m_{10}
 <\frac{5126979}{262144000}<\frac1{50}.
 \label{eq:horizontal-upper}
\end{equation}

For the first ten terms on the left, the same recurrence gives the following
small lower witnesses, listed in increasing \(j\):
\begin{equation}
 50000\,h_jq_j>
 (154,136,121,109,99,91,83,77,71,66),
 \qquad 10\le j\le19.
 \label{eq:horizontal-witnesses}
\end{equation}
Their sum is \(1007\).  Therefore
\[
 \sum_{j\ge10}h_jq_j
 >\sum_{j=10}^{19}h_jq_j
 >\frac{1007}{50000}>\frac1{50},
\]
contradicting \eqref{eq:horizontal-upper}.
This is a direct finite counterexample to equation~(27) on v2 p.~18,
with no estimate of the infinite remainder.

The released script uses the false coefficient: \path{profile.m}, lines
167--172, inserts the resulting factor \code{(3/5)*N} into
\code{BDK}, \code{BD3}, and \code{BD4}.  Thus the observation is
proof-relevant for the \(\xi=-1\) enclosure, rather than a defect in an unused
lemma.

\subsection{Two finite rectangles are not evaluated}
\label{sec:rectangles}

V2 writes the profile/weight pairing as a double series
\[
 \sum_{k\ge1}\sum_{j\ge0}r_{kj},\qquad
 r_{kj}=\mathrm{afm}_k c_j(-1)c_k^{(j)}
       +\mathrm{bfm}_k d_j(-1)c_k^{(j)}.
\]
These are the unmodified symbols printed in v2.  For comparison with the
released implementation, define the corresponding modified summand
\[
 \mathring r_{kj}:=\mathrm{afm}_k\mathring c_j(-1)c_k^{(j)}
       +\mathrm{bfm}_k\mathring d_j(-1)c_k^{(j)}.
\]
For positive modes the two normalizations differ by the fixed factor used in
v2, \(\mathring r_{kj}=(54/11)r_{kj}\).  The missing-ring typo in the printed
double series was recorded in Section~\ref{sec:typos}; the block audit below
uses the modified summands actually present in \path{profile.m}.
From \eqref{eq:profile-seed} and \eqref{eq:transformed-coefficients},
\[
 \mathrm{afm}_1=-\frac{26195}{874},\qquad
 \mathrm{bfm}_1=-\frac{59605}{874}.
\]
All positive-mode weight and moment factors are strictly positive, by
\eqref{eq:minus-one-recurrence} and \eqref{eq:moment-coefficient}.  Hence
every summand in \(\mathcal B_2\) is negative; in particular, that whole
finite block cannot cancel to zero.
The six-region decomposition \((\star)\) on v2 p.~18
(\path{invertibility.tex}, lines 828--836) states that the following two
finite rectangles are computed symbolically:
\begin{align}
 \mathcal B_2&=\{k=1,\ 35\le j\le N_1\},\nonumber\\
 \mathcal B_4&=\{2\le k\le35,\ 35\le j\le N_2\}.
 \label{eq:missing-blocks}
\end{align}
Here \(N_1=5000\) and \(N_2=500\) in the printed proof.

The released \path{profile.m} does not contain either finite sum.
Lines 183--184 use the arrays with range \code{1:35}, which represent only
\(j=0,\ldots,34\).  The remaining variables have the following roles:
\[
\begin{array}{c|c|c}
\text{variable}&\text{intended region}&\text{implemented action}\\ \hline
\code{BD3}&k=1,\ j>N_1&\text{proposed far-row bound}\\
\code{BD4}&2\le k\le35,\ j>N_2&\text{proposed horizontal-tail bound}\\
\code{BD6}&k>35,\ j>35&\text{proposed double-tail bound}
\end{array}
\]
The comment on \code{BD4} calls it region (4), but its formula is a tail
estimate at \(N_2\), so it belongs to region (5).  No variable evaluates
\(\mathcal B_2\) or \(\mathcal B_4\).  The first block has a definite nonzero
contribution, and no enclosure for the second is supplied.  Neither can be
discarded from a two-sided estimate without a separate argument.
For example, direct exact evaluation of its first modified entry gives
\[
 -\frac3{10}<\mathring r_{1,35}<-\frac14,
\]
so the missing first block is not a zero block.

This is distinct from the endpoint typo discussed in
Section~\ref{sec:typos}.  At least one printed endpoint is wrong.  One clean
disjoint convention keeps block (1) through \(j=35\) and starts blocks (2)
and (4) at \(j=36\).  The code's range \code{1:35} instead suggests ending
block (1) at \(j=34\); with that convention the boundary
\(k>35,j=35\) must also be assigned to block (6), or treated separately.
Either endpoint repair is only a partition correction: it does not insert the
two absent finite sums into the implementation.

\subsection{The kernel upper bound has the same finite gap}
\label{sec:kernel-upper-bound}

The kernel value targeted by this calculation is
\[
 \widehat K_1(-1)=1+\sum_{j\ge0}\mathring d_j(-1)c_1^{(j)}.
\]
This follows from the kernel definition on v2 pp.~12--13 and the weight and
moment expansions above.

Line 185 of \path{profile.m} constructs \code{K1pu} from the explicit modes
\(j=0,\ldots,34\) and \code{BDK}, a proposed far-tail allowance based at
\(N_1=5000\).  It contains no separately justified contribution for
\(35\le j<5000\).  The kernel terms are nonnegative, and every omitted
positive-mode term is positive.  The endpoint allowance used by the code is
\[
 \code{BDK}=\frac35N_1c_{N_1}^{(0)}
 \max\{\abs{\mathring c_{N_1}(-1)},\abs{\mathring d_{N_1}(-1)}\},
 \qquad N_1=5000.
\]
A four-term exact witness is enough.  From
\eqref{eq:minus-one-recurrence} and \eqref{eq:moment-coefficient},
\[
 10^5\mathring d_j(-1)c_1^{(j)}
 >(289,278,268,258)\qquad(j=35,36,37,38),
\]
in the corresponding order.  Their sum exceeds \(1093/10^5>1/100\), whereas
exact evaluation of the displayed endpoint expression gives
\(\code{BDK}<1/100\).  Therefore even the first four omitted positive terms
are larger than the entire allowance inserted at line 185.  Accordingly,
\code{K1pu} is not an upper bound for the target kernel value.

The missing rectangles and the false estimate
\eqref{eq:v2-horizontal} are independent.  Either one prevents the released
calculation from certifying the claimed two-sided \(\xi=-1\) pairing.

\section{Large-time comparison and reconstruction}
\label{sec:large-errors}

\subsection{The point-value cascade omits one forcing term}
\label{sec:point-value}

Throughout this subsection, \(\Theta_1\) denotes the modified component
\(\Theta_1^{\rm mod}\); the superscript is suppressed to lighten the formulas.
Recall the normalized average
\[
 \langle h\rangle:=\frac2\pi\int_0^\infty
 \frac{h(\gamma)}{1+\gamma^2}\dd\gamma,
\]
and define the barred first component and its quotient by
\[
 \bar\Theta_1(t,\gamma):=\Theta_1(t,\gamma)-\Theta_1(t,0),\qquad
 G_1=\frac{\bar\Theta_1}{\gamma^2},\qquad
 \widetilde G_1=G_1-G_1(0).
\]
The elementary identity
\(\gamma^2/(1+\gamma^2)=1-(1+\gamma^2)^{-1}\) gives
\begin{align*}
 \langle\bar\Theta_1\rangle
 &=\frac2\pi\int_0^\infty G_1(\gamma)\dd\gamma-\langle G_1\rangle\\
 &=\frac2\pi\int_0^\infty G_1(\gamma)\dd\gamma
   -G_1(0)-\langle\widetilde G_1\rangle.
\end{align*}
Equivalently,
\begin{equation}
 -\langle\bar\Theta_1\rangle
 =G_1(0)+\langle\widetilde G_1\rangle
  -\frac2\pi\int_0^\infty G_1(\gamma)\dd\gamma.
 \label{eq:point-identity}
\end{equation}
Any absolute-value forcing based on this identity must account for three
terms.

The printed discussion on v2 p.~16
(\path{invertibility.tex}, lines 723--729) gives a valid, coarser comparison
using \(\|\bar\Theta_1\|_\infty\).  The released implementation does not use
that right-hand side.  Line 34 of \path{longtime_sym.m} instead drives
\code{Gam20} by
\[
 10\bigl(\code{G1bar}+(2/\code{pil})\code{SG1}\bigr).
\]
According to the legend and the preceding recurrences,
\code{G1bar} bounds the second term of \eqref{eq:point-identity} and
\((2/\code{pil})\code{SG1}\) bounds the integral term.  No allowance for the
first term is included: the single \code{G1bar} contribution is already used
to bound \(\langle G_1-G_1(0)\rangle\).  In the endpoint class,
\(G_1(\gamma)\to0\) as \(\gamma\to\infty\), so one may also derive
\(\abs{G_1(0)}\le\|G_1-G_1(0)\|_\infty\); that requires a second
copy of the same allowance (or a separately proved cancellation).  Line 28
of \path{derivative_sym.m} repeats the omission with the derivative variables.

The omission has two uses: \path{longtime_sym.m} employs it in the large-time
comparison, while \path{derivative_sym.m} supplies the derivative bound used
by the released short-time stepping.  The printed coarse inequality may offer
one repair; alternatively one can propagate a point-value state or add a
justified second oscillation allowance.  None of those repairs is what the
cited code lines execute.

\subsection{Equation (30) reverses both reconstruction signs}
\label{sec:sign-reversed}

Starting from \eqref{eq:filter-start}, split the finite integral at infinity:
\[
 \int_0^te^sT(s)\dd s
 =\widehat T(-1)-\int_t^\infty e^sT(s)\dd s.
\]
Substitution gives the sign-correct identity
\begin{equation}
 \quad
 \Upsilon(t)=
 \bigl(26c_*-\widehat T(-1)\bigr)e^{-t}
 +T(t)+e^{-t}\int_t^\infty e^sT(s)\dd s.
 \quad
 \label{eq:correct-duhamel}
\end{equation}
This is direct algebra from the filter definition; no comparison estimate is
being used.

Equation~(30) on v2 p.~19
(\path{invertibility.tex}, lines 858--863) instead prints
\[
 \bigl(26c_*+\widehat T(-1)\bigr)e^{-t}
 +T(t)-e^{-t}\int_t^\infty e^sT(s)\dd s,
\]
so both the sign of the Laplace datum and the sign of the terminal integral
are reversed.  The first sign also conflicts with the minus sign before the
Laplace term in the overview on v2 p.~7
(\path{invertibility.tex}, line 317).

The direction of the available scalar datum matters.  In
\eqref{eq:correct-duhamel}, a lower bound for
\(26c_*-\widehat T(-1)\) requires an \emph{upper} bound for
\(\widehat T(-1)\).  The v2/CMP lemma, at v2 source lines 849--853, records only
\[
 \widehat T(-1)\ge-45.6.
\]
That lower bound has the wrong one-sided direction for the corrected leading
coefficient.  Thus even if the two signs in equation~(30) began as a
typesetting slip, the released proof and code use the corresponding wrong
direction; the consequence is a certificate error.

\subsection{The code also drops the exponential discount}
\label{sec:discount-dropped}

The final defect is independent of the two signs.  Let
\[
 W(t):=\code{Gamb1}(t)+\code{Gam10}(t)\ge0
\]
be the code's envelope for \(\abs{T(t)}\).  The terminal contribution in
\eqref{eq:correct-duhamel} requires the adverse bound
\begin{equation}
 e^{-t}\int_t^\infty e^sW(s)\dd s.
 \label{eq:discounted-tail}
\end{equation}
Line 40 of \path{longtime_sym.m} computes instead
\[
 e^{-t}\int_t^\infty W(s)\dd s,
\]
with no factor \(e^s\) inside the integral.  Since \(e^s\ge1\), the executed
quantity is smaller than \eqref{eq:discounted-tail}, not an upper bound for
it.  For the elementary test envelope \(W(s)=e^{-2s}\), for example, the code
would produce \(\tfrac12e^{-3t}\), while the required discounted quantity is
\(e^{-2t}\).
Line 39 sets \code{coeff} to \(-45.6\), and line 41 adds
\(\code{coeff}e^{-t}\), following the incorrect plus sign before the Laplace
datum in equation~(30).  Consequently the large-time polynomial produced by
those lines does not certify the lower bound required by the corrected
identity.

\section{Classification and dependency map}
\label{sec:classification}

Table~\ref{tab:classification} summarizes the distinction made throughout
the note.  A direct counterexample means that a finite exact witness
disproves the displayed bound.  An insufficient verification means that the
conclusion may conceivably be true by another argument, but it does not follow
from the check recorded in v2 and retained in CMP.

\begin{longtable}{@{}L{0.30\textwidth}L{0.18\textwidth}L{0.42\textwidth}@{}}
\caption{Classification of the audited items.}\label{tab:classification}\\
\toprule
Location (Loc. in this work) & Classification & Reason \\ \midrule
\endfirsthead
\toprule
Location (Loc. in this work) & Classification & Reason \\ \midrule
\endhead
Lemma 2.2 lower limit \(j=1\) (Section~\ref{sec:lower-limit}) &
\Tmark\ print only &
The underlying expansion in equation (16), and more decisively the later
working modified truncation, already begin at \(j=0\).\\
\addlinespace
Mode-25 array endpoint (Section~\ref{sec:endpoint-shift}) &
\Csmark\ index off by 1 &
Mode-24 values are combined with the mode-25 multiplier; mode 25 is omitted
from the finite ranges.\\
\addlinespace
\path{PQ_sym.m} sign loop (Section~\ref{sec:sign-loop}) &
\Csmark\ insufficient verification: wrong condition &
The loop does not test the coefficient indexed by its loop variable.\\
\addlinespace
Picard upper endpoint (Section~\ref{sec:picard-error}) &
\Cpmark\ order error &
The kernel lower endpoint is used without the needed sign of \(L_1\); the
width term in Lemma~\ref{lem:safe-picard} is absent.\\
\addlinespace
Lemma 5.8 and p.~23 transformed-tail prose (Section~\ref{sec:coef-name}) &
\Tmark\ print only &
The statement uses \(a_N,b_N\), while the released code uses the transformed
coefficients.  The printed uniform sign claim for \(\mathrm{bf}_j\) is false,
but the code branches componentwise; the tail-cutoff subscript is also local.\\
\addlinespace
Appendix B.1 \(L^\infty\) test (Section~\ref{sec:Loo-weak}) &
\Csmark\ insufficient verification: wrong constant &
The script checks a quantity five times smaller than the square of the bound
printed in the proof.\\
\addlinespace
Appendix B.1 value \(113\) (Section~\ref{sec:budget-error}) &
\Cpmark\ arithmetic error &
The preceding half-unit assumptions yield a remainder \(245/46>5\), not
\(3\).\\
\addlinespace
Appendix B.2 extra \(2/\pi\) (Section~\ref{normalization-error}) &
\Tmark\ print only &
The moments are already normalized, and the code adds its tolerance directly
to them.\\
\addlinespace
Appendix B.2 zeroth tail (Section~\ref{sec:moment-error}) &
\Cpmark/\Csmark\ direct counterexample: index off by 1, wrong tail cut-off &
The printed \(c_N^{(-1)}\) is undefined, and eleven omitted terms already
exceed the code's replacement tolerance.\\
\addlinespace
Equation (27) (Section~\ref{sec:equation-27-error}) &
\Cpmark/\Csmark\ direct counterexample to claimed bound &
At \(N=10,k=1\), ten terms on the left exceed \(1/50\), while the asserted
right side is below \(1/50\).\\
\addlinespace
\((\star)\) boundary \(j=35\) (Section~\ref{normalization-error}) &
\Tmark\ print only &
The printed regions overlap.  Either disjoint endpoint convention is local,
but every boundary index must be assigned exactly once.\\
\addlinespace
Weight symbols in the p.~18 double series (Section~\ref{normalization-error}) &
\Tmark\ print only &
The display omits the rings, while the implementation uses
\code{cringp}, \code{dringp}.\\
\addlinespace
Beta/eta statements on pp.~21--23, 26--27 (Section~\ref{sec:beta-tail}) &
\Tmark\ print only &
One denominator omits \(\gamma^2\), and the displayed derivation later has
\(e^{-4r}\) where the statement and code use the independently valid
\(e^{-8Mr}\) decay.  The printed domain, two argument descriptions, and
Appendix A.3.8 notation are also inconsistent with the definitions used.\\
\addlinespace
\((\star)\) finite blocks and \code{K1pu} (Sections~\ref{sec:rectangles} and \ref{sec:kernel-upper-bound}) &
\Csmark\ omitted terms &
Two claimed finite sums are absent, and the kernel upper bound skips a large
positive middle block.\\
\addlinespace
Point-value forcing (Section~\ref{sec:point-value}) &
\Csmark\ omitted term &
The implemented forcing omits \(\abs{G_1(0)}\) from
\eqref{eq:point-identity}.\\
\addlinespace
Equation (30) (Section~\ref{sec:sign-reversed}) &
\Cpmark\ sign and direction &
Both reconstruction signs are reversed, so the available one-sided Laplace
bound is used in the wrong direction.\\
\addlinespace
\path{longtime_sym.m} line 40 (Section~\ref{sec:discount-dropped}) &
\Csmark\ omitted factor &
The factor \(e^s\) required inside the terminal integral is absent.\\
\addlinespace
Local prose/index slips on pp.~11, 20, 27--29, 31 (Section~\ref{sec:local-prose}) &
\Tmark\ print only &
The surrounding formulas or code identify the intended coefficient index,
inequality, iterate index, norm, and summation variable.\\
\bottomrule
\end{longtable}

\subsection{Which v2/CMP conclusions use the proof-relevant items?}

The dependency is more informative than the number of issues.
\begin{center}
\begin{tabular}{@{}L{0.27\textwidth}L{0.62\textwidth}@{}}
\toprule
V2/CMP result or construction & Direct audited dependencies \\ \midrule
Lemma 2.4, positivity of \(P_{25}-Q_{25}\) &
The terminal slot shift and the non-componentwise coefficient-sign loop.\\
\addlinespace
Kernel and source enclosures in Section 5 &
The slot shift propagates into the tail forcing, and the false profile-moment
allowance enters the source enclosure.  The transformed tail is used there in
the componentwise code form identified above.\\
\addlinespace
Lemma 5.9 and the short-time Picard certificate &
The regenerated kernel/source enclosures and the sign-sensitive second
Picard upper bound.\\
\addlinespace
Lemma 4.1 and the derivative enclosure &
The \(L^\infty\) budgets and the omitted point-value forcing; the latter is
also repeated in \path{derivative_sym.m}.\\
\addlinespace
Lemmas 4.4--4.5, the \(\xi=-1\) datum &
The profile moments, false equation (27), omitted finite rectangles, and
unjustified kernel upper bound.\\
\addlinespace
Lemma 4.6, the large-time conclusion &
The comparison inputs above, both signs in equation (30), the wrong
one-sided Laplace direction, and the missing exponential discount.\\
\bottomrule
\end{tabular}
\end{center}

The mode-25 issue illustrates why repairs must be dependency-aware.  Changing
\code{cring(25)} to \code{cring(26)} does not retroactively correct the
already generated tail functions.  A valid proof must form the true
\(P_{25},Q_{25}\), rerun the tail equations, rebuild the kernel and source
bounds, and reevaluate the Picard enclosures.  Likewise, correcting
equation~(30) requires both an upper or two-sided bound on
\(\widehat T(-1)\) and a genuinely discounted terminal-tail estimate.

\subsection{What a valid replacement proof must contain}

There is more than one possible correction, but the preceding dependencies
impose a clear minimum architecture.
\begin{enumerate}
\item The modified-weight recurrence must be truncated through the actual
mode 25, and positivity of the resulting \(P_{25}-Q_{25}\) must be proved
either structurally or by a genuinely componentwise exact check.
\item The profile tails must be controlled by estimates on the omitted
\emph{sums}, not by a single endpoint proxy.  A natural route is to prove
eventual sign and cone invariants for the profile recurrence and combine them
with the closed beta coefficients \eqref{eq:moment-coefficient}.
\item Kernel and source envelopes must then be rebuilt from those corrected
weight and profile inputs.  The Picard comparison must retain the kernel-width
term whenever the lower iterate can be negative, as in
Lemma~\ref{lem:safe-picard}.
\item The \(\xi=-1\) datum needs a complete treatment of every finite and
infinite region.  This may be done by exhaustive directed rectangles, or more
economically by a generating-function identity that replaces the disputed
two-dimensional truncation.
\item The comparison system must include a state controlling the missing
point value \(G_1(0)\), or else use a coarser norm that already controls it.
Finally, the proof must start from \eqref{eq:correct-duhamel}, supply the
correct one-sided or absolute Laplace bound, and estimate the genuinely
discounted tail \eqref{eq:discounted-tail}.
\end{enumerate}
For reproducibility, any remaining machine computation should expose both
the exact construction identities and the directed inequalities consumed by
the proof.  Large partial-fraction numerators or opaque hashes can be retained
as audit metadata, but they should not replace these mathematical interfaces.

\section{Conclusion}

Several discrepancies in arXiv v2 and the published CMP version are ordinary,
readily identifiable typographical slips.  We have kept them separate from the
proof-relevant findings.  The latter show that the symbolic package released
with v2 and retained as CMP's designated companion computation does not,
without additional arguments and recomputation, certify every inequality used
in the all-time positivity proof.  The finite counterexamples in
Proposition~\ref{prop:moment-counterexample}, the discussion of
\eqref{eq:v2-horizontal}, and the finite kernel witness in
Section~\ref{sec:rectangles} show that several numerical enclosures are
actually false as stated.  Other items are insufficient or misdirected checks
whose desired conclusions might still admit different proofs.

The detailed page and source locators are pinned to arXiv v2, while
Section~\ref{sec:cmp-check} verifies item by item that CMP corrects none of them
and designates the same companion archive.  The conclusions therefore apply to
both the v2 certificate and the CMP version of record.  This is not a
counterexample to the intended positivity or invertibility theorem, and it
does not assess the analytic reduction presented in the cited works.  A valid
correction must replace the broken finite implications, propagate the corrected
inputs through every dependent enclosure, and verify the sign-correct
large-time identity.

The lesson is methodological as well as local.  When a CAP substitutes a
finite enclosure for a conceptual step, that enclosure is theorem-critical:
one wrong endpoint may erase the only available implication and invalidate
everything downstream.  CAPs therefore merit unusually explicit construction
identities, independently checkable directed inequalities, and complete
dependency tracking.  In future work, we will present a corrected proof while
simplifying the finite certificates and reducing machine assistance as far as
practical.

\appendix
\section{Reproducibility manifest}
\label{app:manifest}

The individual source hashes used in the audit are listed below.
\begin{center}
\small
\begin{tabular}{@{}L{0.25\textwidth}L{0.65\textwidth}@{}}
\toprule
File & SHA--256 \\ \midrule
\path{invertibility.tex} &
\hash{B10E128371243FE96F737DBA2DA880E3}{06BF681DACE8DBFB74C4E0762D50D536}\\
\path{shorttime_sym.m} &
\hash{FF6CDC9FB72FA3A10A2016A3F1EF7912}{5287C0E63173AFEE751224C4DDF50964}\\
\path{PQ_sym.m} &
\hash{23944665DF4A15E0BEA22425FFB3F6DD}{301485A0EDC263D11CA77105D7D66A30}\\
\path{profile.m} &
\hash{4AE3B6D913F895ABF74D06A685423272}{07F7460CD605B733492882D45E10CA89}\\
\path{build_shorttime_data.m} &
\hash{2BBF2B798D4EA826B4EBF8E724329FCF}{62833FEE8F676FC20BC53673B8512A99}\\
\path{longtime_sym.m} &
\hash{986647667E98A329AF4A3A457A507630}{331D22CCAC86A524F2BDAE973D5FBD71}\\
\path{derivative_sym.m} &
\hash{F57ADAEB706655A2EE0D606C039772A3}{D4FDA962C921196F90126910FB26150B}\\
\bottomrule
\end{tabular}
\end{center}
All line references in the main text refer to these exact files.  The source
archive contains no version-control metadata, so the arXiv version and hashes
are part of the mathematical provenance of this audit.

\section{Minimal exact-arithmetic reproductions}
\label{app:exact}

The two principal counterexamples require only rational recurrence
evaluation.  The following pseudocode spells out the operations; it is
deliberately independent of floating-point arithmetic.

\subsection{Profile and zeroth-moment witness}

\begin{lstlisting}[language=Python]
from fractions import Fraction as Q

a, b = Q(-845,38), Q(-455,19)
profile = {1: (a,b)}
for k in range(2, 512):
    D = 16*k*k + 30*k - 19
    a_new = Q(32*k*k+4*k-101, 2*D)*a + Q(26*(k-1),D)*b
    b_new = Q(-35,D)*a + Q(16*k*k-4*k-12,D)*b
    a, b = a_new, b_new
    profile[k] = (a,b)

c = Q(1)
moment = {1: c}
for k in range(1, 511):
    c *= Q(2*k-1, 2*k)
    moment[k+1] = c

for k in range(501, 512):
    assert profile[k][0] > Q(3,4000)
    assert moment[k] > Q(4,165)
assert sum(profile[k][0]*moment[k] for k in range(501,512)) > Q(1,5000)
\end{lstlisting}

Running the same profile loop to \(N=5000\) verifies
\eqref{eq:rho-bracket} by the two exact comparisons
\[
 \frac3{20}
 <20N^2(a_N^2+b_N^2)
 <\frac4{25}.
\]
The expanded numerator and denominator are not mathematical inputs and need
not be printed.

\subsection{Equation (27) witness}

\begin{lstlisting}[language=Python]
p, q = Q(1,9), Q(11,54)
P, Qmode = {0:p}, {0:q}
for j in range(1, 20):
    if j == 1:
        p = Q(10,17)*q
        q = Q(13,20)*p
    else:
        p = Q(8*(j-1)*p + 10*q, 8*j+9)
        q = Q(8*(j-1)*q + 13*p, 8*j+12)
    P[j], Qmode[j] = p, q

h = Q(1)
H = {0:h}
for j in range(0, 19):
    h *= Q(2*j+1, 2*j+2)
    H[j+1] = h

assert 6*H[10]*max(P[10],Qmode[10]) < Q(1,50)
assert sum(H[j]*Qmode[j] for j in range(10,20)) > Q(1,50)
\end{lstlisting}

\subsection{Omitted kernel block}

Continue the same \code{P}, \code{Qmode}, and \code{H} recurrences through
index 5000.  The
following final lines reproduce both the four-term kernel witness and the
first omitted modified pairing term:
\begin{lstlisting}[language=Python]
ring_factor = Q(54,11)
N1 = 5000
BDK = Q(3,5)*N1*H[N1-1]*ring_factor*max(P[N1],Qmode[N1])
middle = sum(ring_factor*Qmode[j]*H[j] for j in range(35,39))
assert BDK < Q(1,100) < middle

afm1, bfm1 = Q(-26195,874), Q(-59605,874)
r135 = H[35]*ring_factor*(afm1*P[35] + bfm1*Qmode[35])
assert Q(-3,10) < r135 < Q(-1,4)
\end{lstlisting}

These recurrences reproduce the finite witnesses used in this note.  They are
not substitutes for rerunning a complete corrected kernel and source
construction.

\section{Compact source-location index}

\begin{longtable}{@{}L{0.31\textwidth}L{0.25\textwidth}L{0.34\textwidth}@{}}
\toprule
Topic & V2 PDF/source & Released code \\ \midrule
\endfirsthead
\toprule
Topic & V2 PDF/source & Released code \\ \midrule
\endhead
Weight zero mode and truncation &
p.~11, Lemma 2.2; p.~23; p.~26, Appendix A.3.2; source 436--456,
1053--1069, 1185--1193 &
\path{shorttime_sym.m} 23--37, 55--80;
\path{PQ_sym.m} 25--35\\
\addlinespace
Residual sign split &
p.~12, Lemma 2.4; Appendix A.1, p.~24; source 494--500, 1106--1111 &
\path{PQ_sym.m} 52--77\\
\addlinespace
Picard interval &
p.~19, Remark 5.3; p.~20, proof of Lemma 5.1; p.~24, Lemma 5.9;
Appendix A.3.6--A.3.7, p.~27; source 880--918, 1085--1090, 1232--1241 &
\path{shorttime_sym.m} 235--250, 256--289\\
\addlinespace
Transformed-profile tail &
pp.~22--23, Lemma 5.8; source 1019--1066 &
\path{shorttime_sym.m} 98--109\\
\addlinespace
Beta-tail split &
pp.~21--23; Appendix A.3.4, p.~26; Appendix A.3.8, p.~27; source
929, 933, 965--1015, 1080, 1211--1216, 1243--1246 &
\path{shorttime_sym.m} 65, 90, 295--326\\
\addlinespace
Profile \(L^\infty\) bounds &
pp.~27--28; source 1259--1291 &
\path{profile.m} 37--45; \path{longtime_sym.m} 23, 25\\
\addlinespace
Profile moments &
p.~29; source 1314--1369 &
\path{profile.m} 77--109;
\path{build_shorttime_data.m} 55--67\\
\addlinespace
Horizontal tail and rectangles &
p.~18, equation (27) and \((\star)\); source 795--836 &
\path{profile.m} 147--175, 183--186\\
\addlinespace
Kernel \(\widehat K_1(-1)\) enclosure &
pp.~12--13; proof of Lemma 4.4, pp.~17--18; source 506--507, 770--836 &
\path{profile.m} 156--175, 181--186\\
\addlinespace
Point-value comparison &
p.~16; source 723--729 &
\path{longtime_sym.m} 29--35;
\path{derivative_sym.m} 23--29\\
\addlinespace
Large-time reconstruction &
p.~19, equation (30); source 849--863 &
\path{longtime_sym.m} 39--41\\
\addlinespace
Local print-only slips &
pp.~11, 20, 27--29, 31; source 436--440, 910, 1241, 1265, 1349--1351,
1432 &
No independent code defect\\
\bottomrule
\end{longtable}

\end{document}